\documentclass[11pt]{article}
\usepackage[intlimits]{amsmath}
\usepackage{amsfonts, amssymb, amsthm}
\usepackage{esint}
\usepackage{hyperref}
\usepackage{epsfig}
\usepackage{color}
\usepackage{threeparttable} 
\usepackage[toc,page]{appendix}
\newtheorem{thm}{Theorem}[section]
\newtheorem{prop}[thm]{Proposition}

\newtheorem{lemma}[thm]{Lemma}
\newtheorem{defn}[thm]{Definition}

\newtheorem{preremark}[thm]{Remark}

\numberwithin{equation}{section}

\newcommand{\norm}[1]{\left\Vert#1\right\Vert}

\newcommand{\R}{\mathbb R}

\newcommand{\grad} {\nabla}

\DeclareMathOperator{\supp}{supp}

\def\XXint#1#2#3{{\setbox0=\hbox{$#1{#2#3}{\int}$}
       \vcenter{\hbox{$#2#3$}}\kern-.5\wd0}}

\newcommand{\meanbar}[1]{%
\setbox0 = \hbox{$#1 \int$}
\hbox to 0pt{%
\thinspace
\hskip 0.1\wd0
\raise 0.5\ht0
\hbox{%
\lower 0.5\dp0
\hbox{\rule{0.8\wd0}{2\linethickness}}
}%
\hss
}%
}

    \newcounter{myfootertablecounter}

\begin{document}

\title{Optimality  in nonlocal time-dependent obstacle problems}

\author{Ioannis Athanasopoulos, Luis Caffarelli, Emmanouil Milakis}
\date{}
\maketitle

\begin{abstract}
This paper showcases the effectiveness of the quasiconvexity property in addressing the optimal regularity of the temporal derivative and establishes conditions for its continuity in nonlocal time-dependent obstacle problems.

\end{abstract}

AMS Subject Classifications: 35R35, 35R09, 45K05.

\textbf{Keywords}: Obstacle Problem, Parabolic Free Boundary Problems, Non-local operators.

\section{Introduction}\label{sec:intro}

Parabolic obstacle problems can be classified into two main categories, depending on whether the obstacle is time-dependent. This distinction affects both the regularity of the time derivative and the smoothness of the free boundary.

In the time-dependent setting, the coincidence set—that is, the region where the solution coincides with the obstacle—may expand or contract over time. In contrast, in the time-independent case, the coincidence set shrinks monotonically. This monotonicity implies that the time derivative is non-negative and, consequently, continuous. However, in the time-dependent case the time derivative may develop discontinuities, particularly at points where the solution first contacts the obstacle.

Furthermore, since the free boundary evolves in both directions over time, its regularity analysis is, in some respects, more challenging.

A considerable body of literature has been devoted to time-independent obstacle problems. The optimal regularity of the solution, as well as the smoothness of the free boundary in the classical case (also known as the classical one-phase Stefan problem), was established in \cite{C} (see also \cite{CF79}). For the space-nonlocal parabolic operators-i.e. parabolic involving the fractional Laplacian-an analysis of the optimal regularity of the solution and the free boundary was presented in \cite{CF13}. The optimal regularity and the free boundary were obtained in \cite{BFR} for $s>\frac{1}{2}$, where $s$ denotes the order of the fractional Laplacian. The supercritical case $s\in(0,\frac{1}{2})$ along with a generalization to space-nonlocal parabolic operators was addressed in \cite{RT}. Finally, the critical case $s=\frac{1}{2}$ was treated more recently in \cite{FRS} as part of a unified approach.

In this work, we investigate time-dependent obstacle problems with particular emphasis on issues related to temporal regularity. As previously noted, the time derivative of the solution may, in general, exhibit discontinuities. Nevertheless, we observe that the solution detaches from the obstacle in a regular manner; specifically, the positive part of the time derivative remains continuous. In \cite{ACM1}, we established that this positive part of the time derivative is indeed continuous in both the classical "thick" and "thin" obstacle problems (see §§3.1 and 3.2 of \cite{ACM1}, respectively). 

In Section~\S\ref{Optd}, we extend the result of §3.1 in~\cite{ACM1} to the setting of space-nonlocal parabolic operators, which includes, for instance, the fractional Laplacian (cf. the fifth prototype on p.~5 of~\cite{ACM1}). The extension of the result from §3.2 of~\cite{ACM1}, which addresses operators that are nonlocal in both time and space—such as fractional heat operators—will be the subject of future work.
 
In order to proceed with the regularity analysis of the free boundary, it is necessary for the full time derivative to be continuous. In Section §4, for operators similar to those considered in the Section §3, we prove that at points where parabolic density from the past of the coincidence set is positive, the full time derivative is in fact H\"{o}lder continuous. This result is new even in the classical Thick Obstacle Problem.

In the case of nonlocal operators such as the fractional heat operator, a similar result was established in~\cite{ACM2} (see also Section~\S4.2 of~\cite{ACM1} for the classical Thin Obstacle Problem). There, the approach relied on extending the operator to an additional dimension. The case of operators that do not admit such an extension will be addressed in forthcoming work.

\section{Preliminaries. Statement of the problem.}\label{psp}

Let $\Omega\subset\R^n$ be a domain with smooth boundary and assume that $\psi:\R^{n}\times [0,\infty)\rightarrow \R$ be a smooth function with $\psi(x,t)<0$ for $(x,t)\in\Omega^c\times[0,T]$ and $\max\psi(x,0)>0$. 
Find a function $u$ lying above $\psi$, the so called obstacle,  such that
\begin{equation}\label{statem}
\begin{cases}  
\partial_t u-\mathcal{L}u\geq0, \ \ u-\psi\geq 0   &{\rm{in}} \ \  \Omega \times (0,T] \cr
(u-\psi)(\partial_t u-\mathcal{L}u)= 0 &{\rm{in}} \ \  \Omega\times (0,T] \cr
u(x,t)=0 &{\rm{in}} \ \  \Omega^c\times [0,T] \cr 
u(x,0)=\phi(x) &{\rm{on}} \ \  \Omega\times\{0\}. \cr
\end{cases}
\end{equation}
where
\begin{equation}\label{Operator}
\mathcal{L}u(x,t):=\int\limits_{\R^n}(u(y,t)-u(x,t))K(x,y)dy
\end{equation}
with $K$ symmetric in $x$ and $y$ i.e. $K(x,y)=K(y,x)$ for any $x\neq y$ satisfying
\begin{equation}\label{Kbounds}
\frac{\mathbf{1}_{\{|x-y|\leq 2\}}}{\Lambda}\frac{1}{|x-y|^{n+\alpha}}\leq K(x,y)\leq \frac{\Lambda}{|x-y|^{n+\alpha}}
\end{equation}
for some $\Lambda>0$ and $\alpha\in(0,2)$, and $\phi(x)\geq\psi(x,0)$ for $x\in\Omega$. The fractional Sobolev space of order $\alpha/2$ is defined by 
$$H^{\alpha/2}(\Omega)=\bigg\{v\in L^2(\Omega):\frac{|v(x)-v(y)|}{|x-y|^{(\alpha+n)/2}}\in L^2(\Omega\times\Omega)\bigg\}$$
endowed with the norm 
$$||v||^2_{H^{\alpha/2}(\Omega)}=||v||^2_{L^{2}(\Omega)}+(2-\alpha)\iint_{\Omega\times\Omega}\frac{|v(x)-v(y)|^2}{|x-y|^{\alpha+n}}dxdy.$$
The completion of $C^\infty_c(\Omega)$ with $||\cdot||^2_{H^{\alpha/2}(\Omega)}$ is denoted by $H_0^{\alpha/2}(\Omega)$ and $H^{-\alpha/2}$ is the dual of $H_0^{\alpha/2}$.

Weak solutions of \eqref{statem} are obtained as limits of a suitable approximation
procedure. One such procedure is the penalization: one obtains the solution $u$ as a limit of $u^\varepsilon$ as $\varepsilon\rightarrow 0$, where $u^\varepsilon$ solves the problem
\begin{equation}\label{penprob} 
\begin{cases}
\mathcal{L}u^\varepsilon-\partial_t u^\varepsilon=\beta_\varepsilon (u^\varepsilon-\psi^\varepsilon)  &{\rm{in}} \ \  \Omega \times (0,T] \cr
u^\varepsilon(x,t)=0 &{\rm{in}} \ \  \Omega^c\times[0,T] \cr
u^\varepsilon(x,0)=\phi^\varepsilon(x)+\varepsilon &{\rm{on}} \ \ \Omega\times\{0\} 
\end{cases}
\end{equation}
where, for $\varepsilon > 0$, the functions $\phi^\varepsilon$ and $\psi^\varepsilon$ are smooth, $$\beta_\varepsilon(s)=-Ne^{s-\varepsilon}\mathbf{1}_{\{s\leq\varepsilon\}}$$ with $N$ chosen large enough and $$\psi^\varepsilon\rightarrow\psi, \ \ \ \ \ \phi^\varepsilon\rightarrow\phi \ \ \ \ \ \text{uniformly as}\ \ \varepsilon\rightarrow 0.$$

Working with the penalized problem \eqref{penprob}, one can obtain well-posedness—namely existence, uniqueness, and stability under convergence—as in Section 5 of \cite{FKV} (see also \cite{FK} and \cite{ACM1}). Moreover, a priori estimates for the derivatives, i.e. $u^\varepsilon_x$, $u^\varepsilon_t$, and $(u^\varepsilon_{\zeta\zeta})^-$ with $\zeta=(x,t)$, that are independent of $\varepsilon$ follow from the maximum and comparison principles, see \cite{ACM1} in particular for the estimate on $(u^\varepsilon_{\zeta\zeta})^-$ as well as Theorem 3.2 in \cite{ACM2}. Passing to the limit then yields the same bounds for the solution $u$. 

\section{Optimality of time derivative}\label{Optd}
In this section, we demonstrate how quasiconvexity in time enhances the time regularity of the solution, without imposing any assumptions on the behavior of its free boundary. More precisely, we show that the positive time derivative of the solution is always continuous. As mentioned in the Introduction~§\ref{sec:intro}, this result generalizes the findings of §3.1 in~\cite{ACM1} to the framework of space-nonlocal parabolic operators.

Our approach is as follows: we penalize the problem \eqref{statem}, subtract the obstacle from the solution, differentiate the resulting expression with respect to time, and work with the derived equation. We then obtain a global uniform modulus of continuity for the positive part of the time derivative, independent of the $\varepsilon$-penalization. Passing to the limit, this yields the desired continuity result.

The derived problem takes the form:
\begin{equation}\label{derprob} 
\begin{cases}
\mathcal{L}v^\varepsilon -\partial_tv^\varepsilon
=\beta'_\varepsilon(u^\varepsilon
-\psi^\varepsilon)v^\varepsilon+f_t 
& {\rm{in}} \ Q:=\Omega\times (0,T]\\
v^\varepsilon=(-\psi^\varepsilon)_t & {\rm{on}} \ \Omega^c\times (0,T]\\
v^\varepsilon=\mathcal{L}(\phi^\varepsilon-\psi^\varepsilon) & {\rm{on}} \ \  \Omega\times \{0\}.
\end{cases}
\end{equation}
where $v^\varepsilon=\partial_t(u^\varepsilon-\psi^\varepsilon)$ and  $f=-(\mathcal{L} \psi^\varepsilon-\partial_t\psi^\varepsilon)$ and $\mathcal{L}$ is defined in (\ref{Operator}) and (\ref{Kbounds}).
%$$\mathcal{L}u(x,t)=\int\limits_{\R^n}(u(y,t)-u(x,t))K(x,y)dy$$ with $K$ symmetric in $x$ and $y$ i.e. $K(x,y)=K(y,x)$ for any $x\neq y$ and satisfying
%$$\frac{\mathbf{1}_{\{|x-y|\leq 2\}}}{\Lambda}\frac{1}{|x-y|^{n+\alpha}}\leq K(x,y)\leq \frac{\Lambda}{|x-y|^{n+\alpha}}$$ for some $\Lambda>0$ and $\alpha\in(0,2)$

Our method, which uses the approach of \cite{ACM1}, is essentially that of DeGiorgi's, first appeared in his 
celebrated work \cite{DG}. To simplify matters we start with a normalized situation i.e. we assume that the maximum of our solution is one
in the unit parabolic cylinder. We will prove (Proposition \ref{prop1}) that if at the top 
center $v^\varepsilon$ is zero then in a concentric subcylinder into the future $v^\varepsilon$ decreases. 
Then we rescale and repeat. But before that we need several lemmata. Our first lemma asserts that if $v^\varepsilon$ 
is "most of the time" very near to its positive maximum in some cylinder, then in a smaller cylinder into the future $v^\varepsilon$ is strictly positive.   

\begin{lemma}\label{lemma3.1}
Let $Q_1:=B_1\times (-1,0]$ with $B_1:=\{x\in \R^n:|x|\leq 1\}$. 
Suppose that $0<\underset{Q_1}\max~v^\varepsilon\leq 1$ where $v^\varepsilon$ is a solution to~\eqref{derprob}, then there 
exists a constant $\sigma>0$, independent of $\varepsilon$, such that 
\begin{equation}\label{3.2}
\fint\limits_{Q_1}(1-v^\varepsilon)^+dx<\sigma
\end{equation}
implies that $v^\varepsilon\geq 1/2$ in $Q_{1/2}$.
\end{lemma}
\begin{proof} 
For simplicity, we drop the parameter $\varepsilon$.
We begin by deriving an \textit{energy inequality} suited to our needs. To this end, we set $w=1-v$, so that the equation becomes 
$$\mathcal{L} w-\partial_tw=\beta'(u-\psi)(w-1)+f_t.$$
Choose a smooth cutoff function $\zeta$ vanishing near the parabolic boundary of $Q_1$ and $k\geq 0$. Multiply 
the above equation by $\zeta^2(w-k)^+$ and integrate by parts to obtain
$$-\tfrac{1}{2} \int\limits_{-1}^0 D(w,\zeta^2(w-k)^+)dt-\int\limits_{Q_1}(\partial_t w(x,t))(\zeta^2(w-k)^+)dxdt\ \ \ \ \ \ \ \ \ \ \ \ \ \ \ \ \ \ \ \ \ \ \ \ \ \ \ \ \ \ \ \ \ \ \ \ \ \ \ \ \ \ \ \ \ \ \ \ \ \ $$
\begin{equation}\label{Maineq}
\ \ \ \ \ \ \ \ \ \ \ \ \ \ \ \ \ \ \ \ \ \ \ \ \ \ \ \ \ \ \ \ \ \ \ \ \ \ \ \ \ \ \ \ \ \ \ =\int\limits_{Q_1}(\partial_t\beta(u-\psi))(\zeta^2(w-k)^+)dxdt+\int\limits_{Q_1}f_t\zeta^2(w-k)^+dxdt
\end{equation}
where
$$D(u,v):=\iint\limits_{\R^{2n}}(u(x,t)-u(y,t))K(x,y,t)((v(x,t)-v(y,t))dydx,$$
or
$$\tfrac{1}{2} \int\limits_{-1}^0 D(w,\zeta^2(w-k)^+)dt+\tfrac{1}{2}\int\limits_{Q_1}\partial_t[(\zeta(w-k)^+)^2]dxdt\ \ \ \ \ \ \ \ \ \ \ \ \ \ \ \ \ \ \ \ \ \ \ \ \ \ \ \ \ \ \ \ \ \ \ \ \ \ \ \ \ \ \ \ \ \ \ \ \ \ \ \ \ \ \ \ \ \ \ \ $$
\begin{equation}\label{Mmaineq}
\ \ \ \ \ \ \ \ \ \ \ \ \ \ \ \ \ \ \ \ \ \ \ \ =-\int\limits_{Q_1}(\partial_t\beta(u-\psi))(\zeta^2(w-k)^+)dxdt-\int\limits_{Q_1}f_t\zeta^2(w-k)^+dxdt
\end{equation}
Integrating in $t$ the second term on the left and by parts the first term on the right of the above equation we obtain
$$\tfrac{1}{2} \int\limits_{-1}^0 D(w,\zeta^2(w-k)^+)dt+\tfrac{1}{2}\int\limits_{B_1}(\zeta(w-k)^+)^2(x,0)dxdt=-\int\limits_{Q_1}\beta(u-\psi)\partial_t(\zeta^2(w-k)^+)dxdt$$
\begin{equation}\label{maineq}
\ \ \ \ \ \ \ \ \ \ \ \ \ \ \ \ \ \ \ \ \ \ \ \ +\int\limits_{B_1}\beta(u-\psi)\zeta^2(w-k)^+(x,0)dx+\int\limits_{Q_1}((w-k)^+)^2\zeta\partial_t\zeta dxdt-\int\limits_{Q_1}f_t\zeta^2(w-k)^+dxdt
\end{equation}\\
We must estimate each term of the above equation (\ref{maineq}).
\\
\underline{I. Estimate for the integrand of the first term, i.e., $D(w,\zeta^2(w-k)^+)$ in equation (\ref{maineq})}
\\
Observe that
\begin{equation}\label{eq1}D(w,\zeta^2(w-k)^+)=D((w-k)^+,\zeta^2(w-k)^+)+D(-(w-k)^-,\zeta^2(w-k)^+).
\end{equation}
Using the the identity $a^2+b^2=(a-b)^2+2ab$, the first term of (\ref{eq1}) can be written as follows: 
$$D((w-k)^+,\zeta^2(w-k)^+)=D(\zeta(w-k)^+,\zeta(w-k)^+)\ \ \ \ \ \ \ \ \ \ \ \ \ \ \ \ \ \ \ \ \ \ \ \ \ \ \ \ \ \ \ \ \ \ \ \ \ \ \ \ \ \ \ \ \ \ \ \ \ \ \ \ \ \ \ \ \ \ \ \ \ \ \ \ \ \ \ \ \ \ \ \ \ \ \ \ \ \ \ \ \ \ \ \ \ $$ 
%$$\iint\limits_{\R^{2n}}\bigg[(\zeta(w-k)^+)^2(x,t)-(\zeta^2(x,t)+\zeta^2(y,t))(w-k)^+(x,t)(w-k)^+(y,t)+(\zeta(w-k)^+)^2(y,t)\bigg]K(x,y,t)dydx$$
\begin{equation}\label{eq2}
\ \ \ \ \ \ \ \ \ \ \ \ \ \ \ \ \ \ \ \ \ \ \ \ \ \ \ \ \ \ \ \ \ \ \ -\iint\limits_{\R^{2n}}(\zeta(x,t)-\zeta(y,t))^2(w-k)^+(x,t)(w-k)^+(y,t)K(x,y,t)dydx. 
\end{equation}
We estimate the first term of equation (\ref{eq2}) from below,
$$D(\zeta(w-k)^+,\zeta(w-k)^+):=\iint\limits_{\R^{2n}}[(\zeta(w-k)^+ )(x,t)-(\zeta(w-k)^+ )(y,t)]^2K(x,y,t)dydx\ \ \ \ \ \ \ \ \ \ \ \ \ \ \ \ \ \ \ \ \ $$
$$ \ \ \ \ \ \ \ \ \ \\ \ \ \ \ \ \ \ \ \ \ \ \geq\frac{1}{\Lambda}\iint\limits_{\R^{2n}}\frac{\big[(\zeta(w-k)^+ )(x,t)
-(\zeta(w-k)^+ )(y,t)\big]^2}{|x-y|^{n+\alpha}}
\mathbf{1}_{\{|x-y|\leq \frac{1}{2}\}}dydx$$
$$\ \ \ \ \ =\frac{1}{\Lambda}\iint\limits_{\R^{2n}}\frac{[(\zeta(w-k)^+ )(x,t)
-(\zeta(w-k)^+ )(y,t)]^2}{|x-y|^{n+\alpha}}dydx$$
$$\ \ \ \ \ \ \ \ \ \ \ \ \ \ \ \ \ \ \ \ \ \ \ \ \ \ \ \ \ \ \ \ \ \ \ \ \ \ \ \      \ \ \ \ \ \ \ \ \ \ -\frac{1}{\Lambda}\iint\limits_{|x-y|>\frac{1}{2}}\frac{[(\zeta(w-k)^+ )(x,t)
-(\zeta(w-k)^+ )(y,t)]^2}{|x-y|^{n+\alpha}}dydx$$  
using the inequality $(a-b)^2\leq 2(a^2+b^2)$ in the second term, we have
%$$\geq\frac{1}{\Lambda}\iint\limits_{\R^{2n}}\frac{[(\zeta(w-k)^+ )(x,t)-(\zeta(w-k)^+ )(y,t)]^2}{|x-y|^{n+\alpha}}dydx-\frac{2}{\Lambda}\iint\limits_{|x-y|>\frac{1}{2}}\frac{[(\zeta(w-k)^+ )(x,t)]^2+[(\zeta(w-k)^+ )(y,t)]^2}{|x-y|^{n+\alpha}}dydx$$
$$\ \ \ \ \ \ \geq\frac{1}{\Lambda}\iint\limits_{\R^{2n}}\frac{[(\zeta(w-k)^+ )(x,t)-(\zeta(w-k)^+ )(y,t)]^2}{|x-y|^{n+\alpha}}dydx$$  
$$\ \ \ \ \ \ \ \ \ \ \ \ \ \ \ \ \ \ \ \ \ \ \ \ \ \ \ \ \ \ \ \ \ \ \ \ \ \ \ \      \ \ \ \ \ \ \ \ \ \ 
-\frac{2}{\Lambda}\iint\limits_{|x-y|>\frac{1}{2}}\frac{[(\zeta(w-k)^+ )(x,t)]^2+[(\zeta(w-k)^+ )(y,t)]^2}{|x-y|^{n+\alpha}}dydx$$
by definition of fractional Sobolev spaces
$$\ \ \ \ \ \ \ \ \ \ \ \ \ \ \ \ \ \ \ \ \ \ \ \ \ \ \ \ \ \ \ \geq\frac{1}{\Lambda}\norm{\zeta(w-k)^+}^2_{H^\frac{\alpha}{2}}-\frac{4}{\Lambda}\int_{\R^n}(\zeta(w-k)^+)^2(x,t)\bigg(\int\limits_{|y-x|>\frac{1}{2}}\frac{1}{|x-y|^{n+\alpha}}dy\bigg)dx$$
$$\ \ \ \ \ \ \ =\frac{1}{\Lambda}\norm{\zeta(w-k)^+}^2_{H^\frac{\alpha}{2}}-\frac{2^{\alpha+2}\omega_n}{\Lambda\alpha }\int_{\R^n}(\zeta(w-k)^+)^2(x,t)dx$$
On the other hand, we estimate the second term of (\ref{eq2}) from above
$$\iint\limits_{\R^{2n}}(\zeta(x,t)-\zeta(y,t))^2(w-k)^+(x,t)(w-k)^+(y,t)K(x,y,t)dydx\ \ \ \ \ \ \ \ \ \ \ \ \ \ \ \ \ \ \ \ \ \ \ \ \ \ \ \ \ \ \ \ \ \ \ \ \ \ $$
$$\leq 2\int\limits_{B_1}(w-k)^+(x,t)\int\limits_{\R^n}(\zeta(x,t)-\zeta(y,t))^2K(x,y,t)dydx\ \ \ \ \ \ \ \ \ \ \ \ \ \ \ \ \ $$
$$\ \ \ \ \ \ \ \ \ \ \ \ \ \ \ \ \ \ \leq 2\Lambda\int\limits_{B_1}(w-k)^+(x,t)\bigg(\int\limits_{|y-x|>\frac{1}{2}}\frac{2}{|x-y|^{n+\alpha}}dy+\int\limits_{|x-y|\leq\frac{1}{2}}
\frac{|\grad\zeta(x+s_0(y-x)|^2}{|x-y|^{n+\alpha}}dy\bigg)dx $$ 
$$\ \ \ \ \ \ \ \ \ \ \ \ \ \ \ \ \ \ \ \ \ \ \ \ \ \ \ \ \ \ \ \ \ \ \ \ \ \ \ \ \ \ \ \ \ \ \ \ \ \ \ \ \ \ \ \ \ \ \ \ \ \ \ \ \ \ \ \ \ \ \ \ \ \ \ \ \ \ \ \ \ \ \ \ \  \ \ \ \ \ \ \ \ \ \ \ \ \ \ \ \ \ \ \ \ \ \text{for} \ \ s_0\in(0,1)$$
$$\leq C_\alpha\Lambda\int\limits_{B_1}(w-k)^+(x,t)dx.\ \ \ \ \ \ \ \ \ \ \ \ \ \ \ \ \ \ \ \ \ \ \ \ \ \ \ \ \ \ \ \ \ \ \ \ \ \ \ \ \ \ \ \ \ \ \ \ \ \ \ \ \ $$
Finally, the second term of (\ref{eq1}), 
$$D(-(u-k)^-,\zeta^2(u-k)^+)=\hspace{25em}$$
$$ =\iint\limits_{\R^{2n}}((w-k)^-(x,t)(\zeta^2(w-k)^+)(y,t)+(u-k)^-(y,t)(\zeta^2(u-k)^+)(x,t))K(x,y,t)dydx$$
\begin{equation}\label{eq3}
-\iint\limits_{\R^{2n}}(((w-k)^-\zeta^2(w-k)^+)(x,t) +((w-k)^-\zeta^2(w-k)^+)(y,t))K(x,y,t)dydx;
\end{equation}
since the second term in (\ref{eq3}) is zero and using the symmetry of the kernel we have
\begin{equation}\label{eq4} 
D(-(w-k)^-,\zeta^2(w-k)^+)=2\iint\limits_{\R^{2n}}((w-k)^-(x,t)(\zeta^2(w-k)^+)(y,t)K(x,y,t)dydx.
\end{equation}
\underline{II. Estimate for the first term on the right-hand side of (\ref{maineq})}
\\
The crucial estimate for temporal quasilinearity is applied in this first term; consequently, we obtain
$$-\int\limits_{Q_1}\beta(u-\psi)\partial_t(\zeta^2(w-k)^+)dxdt\leq \|\beta\|_{\infty}\Big(\int\limits_{Q_1}(w-k)^+\zeta\partial_t\zeta dxdt+\|(u_{tt})^-\|_\infty\int\limits_{Q_1}\mathbf{1}_{\{u>k\}}dxdt\Big).$$
\underline{III. The estimates for the remaining terms are fairly standard and straightforward. }

We now substitute into equation~\eqref{maineq} the estimates obtained in parts I, II, and III. Note that the upper limit of $t$-integration, $t=0$, can be replaced by any $-1<t\leq 0$. Hence, we have an \textit{energy inequality}

$$\max_{-1\leq t\leq 0}\int\limits_{B_1}(\zeta(w-k)^+)^2dx+\int\limits_{-1}^0\|\zeta(w-k)^+\|_{H^{\frac{\alpha}{2}}}dxdt\ \ \ \ \ \ \ \ \ \ \ \ \ \ \ \ \ \ \ \ \ \ \ \ \ \ \ \ \ \ \ \ \ \ \ \ \ \ \ \ \ \ \ \ \ \ \ \ \ \ \ \ \ \ \ $$

\begin{equation}\label{eq5}
 \ \ \ \ \ \ \ \ \ \ \ \ \ \ \ \ \ \ \ \ \ \ \ \ \ \ \ \ \ \ \ \leq 
C\int_{Q_1}\bigg(([(w-k)^+]^2+(w-k)^+)(|\partial_t\zeta|+1)+\mathbf{1}_{\{w>k\}}\bigg)dxdt
\end{equation}
where $C=C(\|\beta\|_{\infty}, \|(u_{tt})^-\|_{\infty},\|f_t\|_{\infty},\Lambda, \alpha, n)$, and we have omitted term~\eqref{eq3} and the second term in~\eqref{maineq}; they have the correct sign and do not affect the inequality. 

We will obtain a recursive inequality using~\eqref{eq5}. Thus we define for $m=0,1,2,...$ 
$$k_m:=\frac{1}{2}\bigg(1-\frac{1}{2^m}\bigg), \ \ \ R_m:=\frac{1}{2}\bigg(1+\frac{1}{2^m}\bigg)$$
$$Q_m:=\bigg\{(x,t):|x|\leq R_m,\ -R_m^\alpha\leq t\leq 0\bigg\}$$
and the smooth cutoff functions
$$\mathbf{1}_{Q_{m+1}}\leq \zeta_m\leq \mathbf{1}_{Q_m} \ \ \ \ \ \ \ \text{with}\ \ \ \ \ \ \  |\grad \zeta_m|\leq C2^m,\ \ |\partial_t\zeta_m|\leq C4^m.$$
Substituting $\zeta=\zeta_m$ and setting $w_m=(w-k_m)^+$ in the \textit{energy inequality}
~\eqref{eq5} we obtain, by using the Sobolev inequality, that
$$\bigg(\int_{Q_m}(\zeta_mw_m)^{2\frac{n+\alpha}{n}}dxdt\bigg)^{\frac{n}{n+\alpha}}\leq C\bigg( (C4^m+1)\Big[\int_{Q_m}w_m^2dxdt+\int_{Q_m}w_mdxdt\Big]+|Q_m\cap\{w_m\neq 0\}|\bigg)$$
$$\leq  C\bigg(4^mC\int_{Q_m}w_m^2dxdt+|Q_m\cap\{w_m\neq 0\}|\bigg).$$
Since $$(k_m-k_{m-1})^2|Q_m\cap\{w_m\neq 0\}|\leq \int_{Q_m}w^2_{m-1}dxdt$$
we obtain
\begin{eqnarray}\label{eq6}
\int(\zeta_mw_m)^2dxdt&\leq& \bigg(\int(\zeta_mw_m)^{2\frac{n+\alpha}{n}}dxdt\bigg)^{\frac{n}{n+\alpha}}\bigg|Q_m\cap\{w_m\neq 0\}\bigg|^{\frac{\alpha}{n+\alpha}}\nonumber\\
&\leq& CA^{m-1}\bigg(\int(\zeta_{m-1}w_{m-1})^2dxdt\bigg)^{1+\frac{\alpha}{n+\alpha}}
\end{eqnarray}
where $A>1$. %$A=16$

Setting $$I_m:=\int(\zeta_mw_m)^2dxdt$$
we have our the recursive inequality
\begin{equation}\label{recurineq}
I_m\leq CA^{m-1}I_{m-1}^{1+\frac{\alpha}{n+\alpha}}.
\end{equation}
Hence, by induction $$I_m\leq A^{-\frac{n+\alpha}{\alpha}m}I_0$$ provided $$I_0\leq \frac{1}{A^{(\frac{n+\alpha}{\alpha})^2}C^{\frac{n+\alpha}{\alpha}}}=:\sigma$$
Consequently, $$\ \ \ \ \ \ \ \ \ \ \ \ \ \ \ \ \ \ \ \  I_m\longrightarrow 0\ \ \ \ \ \ \text{as}\ \ \  m\rightarrow 0.$$ 
\end{proof}

In the following Lemma \ref{lemma3.2}, we address the more delicate alternative case to that of Lemma \ref{lemma3.1}.

\begin{lemma}\label{lemma3.2} Under the same hypotheses as in Lemma~\ref{lemma3.2}, if
\begin{equation}\label{eq3.2}
\fint\limits_{Q_1}(1-v^\varepsilon)^+dxdt\geq\sigma
\end{equation}
 %Suppose $v^\varepsilon$ solves~\eqref{derprob} and satisfies $0<\underset{Q_1}\max~v^\varepsilon\leq 1$.
then there exists a constant $C>0$, independent of $\varepsilon$, such that $(v^\varepsilon)^+\leq 1-C\sigma$ in $Q_{1/2}$.
%if $v^\varepsilon(0,0)=0$, then $v^\varepsilon\leq 1-C\sigma$ in $Q_{1/2}$ where $\sigma$ is the constant from Lemma~\ref{lemma3.1}.
\end{lemma}

\begin{proof} 
Again, we drop the parameter $\varepsilon$. 
%Since $v(0,0)=0$, by Lemma~\ref{lemma3.1}
%\begin{equation}\label{3.2}
%\fint\limits_{Q_1}(1-v^\varepsilon)^+dxdt\geq\sigma
%\end{equation}
%It follows then that %there exists a constant $c_0<1$ such that 
Observe that~\eqref{eq3.2} implies that 
$$\big|\big\{v<1-\tfrac{\sigma}{4}\big\}\cap Q_1\big|\geq \tfrac{1}{4}\sigma|Q_1|.$$
Therefore, we set $$w:=\tfrac{4}{\sigma}\big[v-(1-\tfrac{\sigma}{4})\big]$$
and we see that, since $\beta'\geq0$, $w$ satisfies
\begin{equation}\label{subsol1}
\mathcal{L} w-\partial_t w\geq\beta'(u-\psi)w+f_t.
\end{equation} 
Then, following the same approach as in the proof of the previous Lemma~\ref{lemma3.1}, we multiply both sides of the inequation~\eqref{subsol1} by $\zeta^2(w-k)^+$ and integrate by parts to obtain
$$\tfrac{1}{2} \int\limits_{-1}^0 D(w,\zeta^2(w-k)^+)dt+\int\limits_{Q_1}(\partial_t w(x,t))(\zeta^2(w-k)^+)dxdt\ \ \ \ \ \ \ \ \ \ \ \ \ \ \ \ \ \ \ \ \ \ \ \ \ \ \ \ \ \ \ \ \ \ \ \ \ \ \ \ \ \ \ \ \ \ \ \ \ \ $$
\begin{equation}\label{ineq}
\ \ \ \ \ \ \ \ \ \ \ \ \ \ \ \ \ \ \ \ \ \ \ \ \ \ \ \ \ \ \ \ \ \ \ \ \ \ \ \ \ \ \ \ \ \ \ \leq-\int\limits_{Q_1}\beta'(u-\psi))(\zeta (w-k)^+)^2dxdt-\int\limits_{Q_1}f_t\zeta^2(w-k)^+dxdt
\end{equation}
and, since the first term on the right-hand side is nonpositive,
\begin{equation}\label{mainineq}
\tfrac{1}{2} \int\limits_{-1}^0 D(w,\zeta^2(w-k)^+)dt+\int\limits_{Q_1}(\partial_t w(x,t))(\zeta^2(w-k)^+)dxdt\leq -\int\limits_{Q_1}f_t\zeta^2(w-k)^+dxdt.
\end{equation}
Estimating the terms as in the previous Lemma~\ref{lemma3.1} we arrive at a similar \textit{energy inequality}:

$$\max_{-1\leq t\leq 0}\int\limits_{B_1}(\zeta(w-k)^+)^2dx+\int\limits_{-1}^0\|\zeta(w-k)^+\|_{H^{\frac{\alpha}{2}}}dxdt\ \ \ \ \ \ \ \ \ \ \ \ \ \ \ \ \ \ \ \ \ \ \ \ \ \ \ \ \ \ \ \ \ \ \ \ \ \ \ \ \ \ \ \ \ \ \ \ \ \ \ \ $$
\begin{equation}\label{enineq2} \ \ \ \ \ \ \ \ \ \ \ \ \ \ \ \ \ \ \ \ \ \ \ \ \ \ \ \ \ \ \ \leq 
C\int\limits_{Q_1}\bigg([(w-k)^+]^2|\partial_t\zeta|+(w-k)^+\bigg)dxdt
\end{equation}
where the constant $C=C(\|f_t\|_{\infty},\Lambda, \alpha, n)$ depends only on the indicated parameters.

By repeating the steps of the previous Lemma~\ref{lemma3.1}, we obtain a recursive inequality similar to~\eqref{recurineq}. Nevertheless, as we have seen, in order for the corresponding quantity $I_m$ to converge to $0$,  the $L^2$-norm of our $w^+$ must be sufficiently small. Nevertheless, in our case, %this does not meet the required smallness condition; specifically, 
$$\fint\limits_{Q_1}(w^+)^2dxdt\leq 1-\tfrac{1}{4}\sigma$$ 
which is not small enough to ensure convergence. 
To overcome this difficulty, we employ the De Giorgi isoperimetric lemma (Lemma~\ref{lemma3.3}), adapted to our setting. While its proof is similar to that of Lemma 5.1 in \cite{ACM2}, we include it here for completeness and to highlight the differences.
\begin{lemma}\label{lemma3.3}
Given an $\eta>0$ there exists a $\delta>0$, $\mu>0$, and $\lambda\in (0,1)$, depending only on $\alpha$, $\Lambda$, and $n$, such that for any bounded $w$ , $w\leq 1$, satisfying~\eqref{subsol1} in $Q_1$ and 
$$|\{(x,t)\in Q_1 : w\leq 0\}|\geq \mu|Q_1|,$$ 
if 
$$|\{(x,t)\in Q_1:0<w<1-\lambda\}|<\delta|Q_1|$$ 
then 
$$\fint\limits_{Q_{1/2}}\Big(\big(w-(1-\lambda^2)\big)^+\Big)^2dxdt<\eta.$$
\end{lemma}
\begin{proof}
Suppose the conclusion is false, then for some $\eta>0$
$$0<\eta\leq \fint\limits_{Q_{1/2}}\Big(\big(w-(1-\lambda^2)\big)^+\Big)^2dxdt\leq\frac{|Q_{1/2}\cap\{w>1-\lambda^2\}}{|Q_{1/2}|}.$$
Therefore there exists a $t_0>-\frac{1}{2}$ such that
$$|\{x\in B_{1/2}:w(x,t_0)>1-\lambda^2\}|>\frac{\eta}{2}|Q_{1/2}|.$$
Moreover for $\mathbf{1}_{B_{1/2}}\leq\xi(x)\leq \mathbf{1}_{B_1}$ and for $\lambda'=\sqrt{2}\lambda$
$$E(t_0):=\int\limits_{\R^n}[\xi(x)(w(x,t_0)-(1-(\lambda')^2))^+]^2dx$$$$\ \ \ \ \ \ \ \ \ \ \ \ \geq((\lambda')^2-(\lambda^2))^2|B_{1/2}\cap\{w>1-\lambda^2\}|$$
$$\geq(\lambda')^4\frac{\eta}{8}|Q_{1/2}|)\ \ \ \ \ \ \ \ \ \ \ \ \ \ \ \ \ $$

On the other hand, we multiply both sides of \eqref{subsol1} by $\zeta^2((w-(1-(\lambda')^2)^+$  where $\zeta(x,t):=\xi (x)\tau (t)$, with $\xi$ as above and 
$\mathbf{1}_{[-1/2,0]}\leq\tau(t)\leq\mathbf{1}_{[-1,0]}$, and integrate by parts to have
\begin{equation}\label{mainineq1}
\frac{1}{2} \int\limits_{-1}^0 D(w,\zeta^2(w-(1-(\lambda')^2)^+)dt+\int\limits_{Q_1}(\partial_t w(x,t))(\zeta^2(w-(1-(\lambda')^2)^+)dxdt$$$$\leq -\int\limits_{Q_1}f_t\zeta^2(w-(1-(\lambda')^2)^+dxdt.
\end{equation}
Then as in Lemma~\ref{lemma3.1} we obtain
$$\int\limits_{\R^n}(\zeta(w-(1-(\lambda')^2)^+)^2(x,0)dx+\int\limits_{-1}^0  D(\zeta(w-(1-(\lambda')^2)^+,\zeta(w-(1-(\lambda')^2)^+)dt$$ 
$$+2\int\limits_{-1}^0\iint\limits_{\R^{2n}}((w-(1-(\lambda')^2))^-(x,t)(\zeta^2((w-(1-(\lambda')^2))^+)(y,t)K(x,y,t)dydxdt $$
$$\leq\int\limits_{-1}^0\iint\limits_{\R^{2n}}(\zeta(x,t)-\zeta(y,t))^2(w-(1-(\lambda')^2))^+(x,t)(w-(1-(\lambda')^2))^+(y,t)K(x,y,t)dydxdt$$
\begin{equation}\label{mainineq2}
+\int\limits_{-1}^0\int\limits_{\R^n}((w-(1-(\lambda')^2))^+)^2(\zeta^2)_tdxdt-\int\limits_{\R^n}f_t\zeta^2(w-(1-(\lambda')^2)^+dxdt.   
\end{equation}
Observe that the three terms on the right-hand side of~\eqref{mainineq2} are each controlled by $C(\lambda')^4$, where the constant $C=C(\|f_t\|_{\infty},\Lambda, \alpha, n)$. Since all three terms on the left-hand side of~\eqref{mainineq2} are nonnegative, we may omit the first two and retain only the third. Notice that this is the opposite approach to the one taken in deriving the \textit{energy inequality} in Lemma~\ref{lemma3.2}, where we instead focused on the first two terms. Therefore, since $\tau(t)\equiv1$ for $\in [-\tfrac{1}{2},0]$, we have
\begin{equation}\label{mainineq3}
\int\limits_{-1}^0\iint\limits_{\R^{2n}}\big[w-(1-(\lambda')^2)\big]^-(x,t)\Big(\xi^2\big[w-(1-(\lambda')^2)\big]^+\Big)(y,t)K(x,y,t)dydxdt\leq C(\lambda')^4|Q_1|.
\end{equation}
Set $I:=\big\{t\in[-1,0]:|\{w(.,t)\leq0\}\cap B_1|\geq\frac{1}{2}\mu|Q_1|\big\}$ then, since, $\inf_{|x-y|\geq2}K(x,y,t)\geq C\Lambda^{-1}$, the left-hand side of~\eqref{mainineq3} is bounded below by
$$\frac{C(1-(\lambda')^2)\mu|Q_1|}{2\Lambda}\int\limits_I\int\limits_{\R^n}\big[\xi^2((w-(1-(\lambda')^2))^+\big](y,t)dydt.$$ 
Hence, by choosing $\lambda'<\frac{1}{\sqrt2}$ and using $(w-(1-(\lambda')^2)^+\leq(\lambda')^2$, we obtain
\begin{equation}\label{mainineq4}
\int\limits_I\int\limits_{\R^n}\big((\xi(w-(1-(\lambda')^2)^+\big)^2(y,t)dydt\leq\frac{C}{\mu|Q_1|}(\lambda')^6.
\end{equation}
Furthermore take $\lambda'\leq(\mu|Q_1|/C)^8$. Since $|I|\geq\frac{\mu|Q_1|}{2|B_1|}$ by choosing a set $F\subset I$ with $|F|<(\lambda')^\frac{1}{8}$ we have
$$\int\limits_{\R^n}(\xi(w-(1-(\lambda')^2)^+)^2(y,t)dy\leq (\lambda')^{6-\frac{1}{3}}\ \ \ \forall \ t\in I\setminus F .$$
Moreover for $(\lambda')^{2-\frac{1}{3}}\leq\frac{\eta}{16}|Q_{1/2}|$ and $\forall\ t\in I\setminus F$
$$E(t)=\int\limits_{\R^n}(\xi(w-(1-(\lambda')^2)^+)^2(y,t)dy\leq(\lambda')^4\frac{\eta}{16}|Q_{1/2}|$$

Now, let $t^*<t_0$ ($t_0$ as above) be the first time for which 
$E(t^*)\leq(\lambda')^4\frac{\eta}{16}|Q_{1/2}|$. Since $E(t_0)\geq(\lambda')^4\frac{\eta}{8}|Q_{1/2}|)$ and $\frac{d}{dt}E(t)\leq C(\lambda')^4$, the energy must decrease by at least $\Delta E=(\lambda')^4\frac{\eta}{16}|Q_{1/2}|$ 
over the time interval $J:=(t^*,t_0)$
%$$J:=\{(t^*,t_0): (\lambda')^4\frac{\eta}{16}|Q_{1/2}\}\leq(\lambda')^4\frac{\eta}{8}|Q_{1/2}|)\}$$
of length 
$|J|\geq\frac{\Delta E}{C(\lambda')^4}=\frac{\eta}{16C}|Q_{1/2}|$. Notice the estimate $\frac{d}{dt}E(t)\leq C(\lambda')^4$ follows in a similar manner as in~\eqref{mainineq2} except that we now multiply the equation by
$(\xi(w-(1-(\lambda')^2)^+)^2$ and integrate by parts only in space. In the resulting inequality, one of the terms on the left-hand side is precisely 
$E'(t)$ while the remaining two terms are non-negative and can therefore be omitted. The right-hand side remains unchanged.

For $t\in J\cap N$ where $N:=\{t\in[-1,0]:|\{w(.,t)\leq0\}\cap B_1|\geq\frac{1}{2}\mu|Q_1|$ we have
$$C(\lambda')^4 |Q_1|\geq\int\limits_{-1}^0\iint\limits_{\R^{2n}}\big[w-(1-(\lambda')^2)\big]^-(x,t)\Big(\xi^2\big[w-(1-(\lambda')^2)\big]^+\Big)(y,t)K(x,y,t)dydxdt $$$$\geq\frac{C\mu|Q_1|}{2\Lambda}\int\limits_{J\cap N}\int\limits_{\R^n}\xi^2(x)\big[w-(1-(\lambda')^2)\big]^+dxdt\ \ \ \ \ \ \ \ \ \ \ \ \ \ \ \ \ \ \ \ \ \ \ \ \ \ $$$$\ \ \ \ \ \ \geq\frac{C\mu|Q_1|}{4(\lambda')^2\Lambda}\int\limits_{J\cap N}\int\limits_{\R^n}\big(\xi(x)\big[w-(1-(\lambda')^2)\big]^+\big)^2dxdt=\frac{C\mu|Q_1|}{4(\lambda')^2\Lambda}\int\limits_{J\cap N}E(t)dt $$$$\geq\frac{C\mu^2(\lambda')^2|Q_1||Q_{1/2}||J\cap N|}{64\Lambda}.\ \ \ \ \ \ \ \ \ \ \ \ \ \ \ \ \ \ \ \ \ \ \ \ \ \ \ \ \ \ \ \ \ \ \ \ \ \ \ \ \ \ \ \ \ \ $$
Therefore,
$$|J\cap N|\leq\overline{C}\frac{(\lambda')^2}{\mu^2}$$
and for $(\lambda')^2\leq(\mu^2|J|/(2\overline{C}$  implies
$$|J\cap N|\leq\frac{|J|}{2}.$$
Hence, for every $t\in J\setminus N$
$$M(t):=|\{0<w(.,t)<1-\lambda\}|\geq(1-\frac{1}{2}\mu|Q_1|-\frac{1}{2}\eta|Q_{1/2}|>\frac{1}{2}$$
for $\mu=\eta$ and
$$|\{(x,t)\in Q_1:0<w<1\}|\geq\int\limits_{-1}^0M(t)\,dt\geq\int\limits_{J\setminus N}M(t)\,dt\geq\frac{|J|}{4}\geq\frac{\eta}{16}|Q_1|$$
a contradiction for $\delta=\frac{\eta}{16}$.
\end{proof}
Having the DeGiorgi isoperimetric lemma (Lemma \ref{lemma3.3}) at our disposal  we proceed by induction; we set$$\ \ \ \ \ \ \  \ \ \ \ \ \  \ \ w_{k+1}=\frac{1}{\lambda^2}(w_k-(1-\lambda^2)), \ \ \ \ \ \ w_0=w\ \ \ \ \ \ \ \ \ \ \ \text{for} \ \ \  k=0,1,2,...$$
which satisfy \eqref{subsol1}. We will show that in finite number of steps $k_0=k_0(\delta)$ (where $\delta$ is defined as in Lemma \ref{lemma3.3}),
$$|\{w_{k_0}>1-\lambda \}|=0.$$
Indeed, if,  for $k=0,1,...,k_0$, we have $|\{0<w_k<1-\lambda\}\cap Q_1|\geq\delta|Q_1|$  then
$$|\{w_k>1-\lambda\}\cap Q_1|=|\{w_k>0\}\cap Q_1|-|\{0<w_k<1-\lambda\}\cap Q_1|$$$$\ \ \ \ \leq |\{w_k>0\}\cap Q_1|-\delta |Q_1|$$$$\ \ \ \ \ \ \ \ \ \ \ \ \ \ \ \leq|\{w_{k-1}>1-\lambda^2\}\cap Q_1|-\delta |Q_1|$$$$\ \ \ \ \ \ \ \ \ \ \ \ \ \leq|\{w_{k-1}>1-\lambda\}\cap Q_1|-\delta |Q_1|$$$$\ \ \ \ \ \ \leq |\{w_0>0\}\cap Q_1|-k\delta |Q_1|$$$$\ \ \ \ \ \ \ \ \ \ \ \ \ \ \ \ \ \ \ \ \leq 0\ \ \ \ \ \ \ \ \ \ \ \ \ \ \ \ \ \ \ \ \ \ \ \ \ \ \ \ \ \ \ \text{if}\ \ k\geq\frac{1}{\delta}.$$
On the other hand, if there exists an index $k'$, $0\leq k'\leq k_0$, such that $|\{0<w_{k'}<1-\lambda\}\cap Q_1|\leq\delta|Q_1|$ then we apply Lemma \ref{lemma3.3} to $w_{k'}$ to make the average of $(w_{k'+1}^+)^2$ as small as we wish.

We now observe that $\overline{w}:=w_{k'+1}$ satisfies \eqref{subsol1} and the  corresponding \textit{energy inequality} \eqref{enineq2}. Therefore, repeating the steps of Lemma~\ref{lemma3.1} we obtain a recursive inequality analogous to \eqref{recurineq} with the corresponding $I_m$ converging to $0$ as  $m\rightarrow\infty$. Consequently,
$$v\leq 1-(\lambda')^{2(k'+1)}\tfrac{\sigma}{4}.$$
\end{proof}
We shall next iterate the above results over a dyadic sequence of shrinking cylinders in order to obtain the continuity of $(v^\varepsilon)^+$, independent of $\varepsilon$.
\begin{prop}\label{prop1}
Let $v^\varepsilon$ be a solution to \eqref{subsol1} in $Q$ then, for any $(x,t), (x_0,t_0)\in Q$,
$$\big|(v^\varepsilon)^+(x,t)-(v^\varepsilon)^+(x_0,t_0)\big|\leq C\omega(|x-x_0|^{\alpha}+|t-t_0)|)$$ 
where $C$ is independent of $\varepsilon$, and $\omega$ denotes the modulus of continuity.
\end{prop}
\begin{proof}
It is enough to take the points $(x,t), (x_0,t_0)$ of distance between them less than one and also  $(v^\varepsilon)^+(x_0,t_0)=0$. We drop again $\varepsilon$ and take $(x_0,t_0)=(0,0),\ Q=Q_1$. Set
$$Q_k:=Q_{R_k},\ \ \ \ \ M_k:=\sup_{Q_k} v $$
where $R_k
%:=r_kR$, $r_k
:=\frac{\sigma}{8}M_k$ and $$\bar{v}:=\frac{v_k}{M_k}$$
where $v_k(x,t):=v(R_kx,(R_k)^\alpha t)$. Then $\bar{v}$ satisfies \eqref{subsol1} and therefore by Lemma \ref{lemma3.2}, $$\sup_{Q_{R'}}\bar{v}\leq 1-C\sigma$$
or in our original setting
$$\sup_{Q_{k+1}}v\leq \mu_k\sup_{Q_k}v$$
where $\mu_k=1-C(\sup_{Q_k}v^+)^{1+\frac{n}{2}}$. So, even, if $\mu_k\rightarrow 1$ as $k\rightarrow \infty$, $M_k\rightarrow 0$. 
To finish the proof, we use a standard barrier argument to get the continuity from the future. 
\end{proof}

\begin{thm}\label{ThmOptiCont}
Let $u$ be a solution to (\ref{statem}) then $(u-\psi)_t^+$ is continuous. 
%with logarithmic modulus of continuity. 
\end{thm}
\begin{proof}
It is well known that a subsequence of $v^\varepsilon$ will converge uniformly to the unique solution of (\ref{statem}).
\end{proof}
\section{H\"{o}lder continuity of time derivative}\label{hole}
In general, one does not expect the time derivative $u_t$ to be continuous, and a discontinuity may appear when the solution first comes into contact with the obstacle$-$that is, $u_t^-$ may experience a jump. On the other hand, if one assumes that the solution has already sufficiently “coincided” with the obstacle in the past then $u_t^-$ is indeed continuous. The purpose of this section is precisely to establish this continuity. In fact, we will show that the full time derivative $u_t$ is H\"older continuous. It is worth noting that this result does not rely on the time quasiconvexity. The following result will be used in the proof of Theorem~\ref{hthm} below.

\begin{lemma}\label{hlem1}
Let $u$ be the solution to our problem then $\frac{\partial u}{\partial t}\in L^2\big((0,T]:H^{\alpha/2}(\Omega)\big)$.
\end{lemma}
\begin{proof}
Penalize the problem \eqref{statem} and look at the derived equation \eqref{derprob}. Choose a smooth cutoff function $\zeta$ vanishing near the parabolic boundary of $Q_1:=B_1\times (-1,0]$. 
 Therefore in the weak formulation of \eqref{derprob} with $\eta:=\zeta^2(v^\varepsilon)$ as a test function we have
\begin{equation}\label{hlmaineq}
\int\limits_{-1}^0\int\limits_{\R^n}(\mathcal{L}v^\varepsilon-\partial_t v^\varepsilon)\zeta^2v^\varepsilon dxdt=\int\limits_{-1}^0 \int\limits_{\R^n}(\beta'_\varepsilon(u^\varepsilon
-\psi^\varepsilon)v^\varepsilon+f_t)\zeta^2v^\varepsilon dxdt.
\end{equation}
For simplicity, we drop the  $\varepsilon$ superscript. Using the nonnegativity of the first term on the right hand side we arrive at the following: 
\begin{equation}\label{hlineq1}
\int\limits_{-1}^0\Bigg[\frac{1}{2}D(v,\zeta^2v)+\int\limits_{\R^n}((\partial_tv)(\zeta^2v))(x,t)dx\Bigg]dt\leq -\int\limits_{-1}^0\int\limits_{\R^n}f_t\zeta^2vdxdt  
\end{equation}
where
$$D(u,v):=\iint\limits_{\R^{2n}}(u(x,t)-u(y,t))K(x,y,t)((v(x,t)-v(y,t))dydx.$$

Then, we write
$$D(v,\zeta^2v)=\iint\limits_{\R^{2n}}K(x,y,t)\Big[(\zeta v)^2(x,t)+(\zeta v)^2(y,t)-(\zeta^2(x,t)+\zeta^2(y,t))v(y,t)v(x,t)\Big])dydx$$
$$\ \ \ \ \ \ \ \ \ \ \ \ \ \ \ \ \ =\iint\limits_{\R^{2n}}K(x,y,t)\Big[\big((\zeta v)(x,t)-(\zeta v)(y,t)\big)^2-(\zeta(x,t)-\zeta(y,t))^2v(y,t)v(x,t)\Big])dydx$$ 
\begin{equation}\label{hleq2}
=\iint\limits_{\R^{2n}}K(x,y,t)\big((\zeta v)(x,t)-(\zeta v)(y,t)\big)^2dydx-\iint\limits_{\R^{2n}}K(x,y,t)(\zeta(x,t)-\zeta(y,t))^2v(y,t)v(x,t))dydx
\end{equation}

The second term of (\ref{hlineq1}) becomes
\begin{equation}\label{hleq4}
\int\limits_{-1}^0\int\limits_{\R^n}((\partial_tv)(\zeta^2v))(x,t)dxdt=\frac{1}{2}\int\limits_{-1}^0\int\limits_{\R^n}\partial_t((\zeta v)^2)dxdt-\int\limits_{-1}^0\int\limits_{\R^n}\zeta\zeta_tv^2dxdt
\end{equation}

Inserting (\ref{hleq2}) and (\ref{hleq4}) into (\ref{hlineq1}) and rearranging properly we obtain
$$\frac{1}{2}\int\limits_{-1}^0\iint\limits_{\R^{2n}}K(x,y,t)\big((\zeta v)(x,t)-(\zeta v)(y,t)\big)^2dydxdt+\frac{1}{2}\int\limits_{-1}^0\int\limits_{\R^n}\partial_t((\zeta v)^2)dxdt\ \ \ \ \ \ \ \ \ \ \ \ \ \ \ \ \ \ \ \ \ \ \ \ \ \ \ $$
$$\ \ \ \ \ \ \ \ \leq\int\limits_{-1}^0\iint\limits_{\R^{2n}}K(x,y,t)(\zeta(x,t)-\zeta(y,t))^2v(y,t)v(x,t)dydxdt+\int\limits_{-1}^0\int\limits_{\R^n}\zeta\zeta_tv^2dxdt-\int\limits_{-1}^0\int\limits_{\R^n}f_t\zeta^2vdxdt$$
or
$$\sup_{-1\leq t\leq 0}\int\limits_{\R^n}(\zeta v)^2dx+\int\limits_{-1}^0\iint\limits_{\R^{2n}}\frac{\big((\zeta v)(x,t)-(\zeta v)(y,t)\big)^2}{|x-y|^{n+\alpha}}dxdydt \ \ \ \ \ \ \ \ \ \ \ \ \ \ \ \ \ \ \ \ \ \ \ \ \ \ \ \ \ \ \ \ \ \ \ \ \ \ \ \ \ \ \ \ \ \ \ \ \ \ \ \ \ \ \ \ \ \ \ \ \ \ \ \ \ \ \ \   $$
$$ \ \ \ \ \ \ \ \ \ \ \ \ \ \  \leq \bigg[\int\limits_{-1}^0\iint\limits_{\R^{2n}}K(x,y,t)(\zeta(x,t)-\zeta(y,t))^2v(y,t)v(x,t)dydxdt$$
\begin{equation}\label{hlineq2}
\ \ \ \ \ \ \ \ \ \ \ \ \ \ \ \ \ \ \ \ \ \ \ \ \ \ \ \ \ \ \ \ \ \ \ \ \ \ \ \ \ \ \ \ \ \ \ \ \ \ \ \ \ \ \ \ \ \ \ \ \ \ \ \ \ \ \ \ \ \ \ \ \ \ \ \ \ \ \ \ \ \ \ \ \ \ \ \ +2\int\limits_{-1}^0\int\limits_{\R^n}\zeta\zeta_tv^2dxdt+2\int\limits_{-1}^0\int\limits_{\R^n}\zeta^2|f_tv|dxdt\bigg]  
\end{equation}
where the constant $C$ depends on $\Lambda$ and $n$.
 
Now, the first term on the right above due to its $xy$-symmetry can be written as
$$\int\limits_{-1}^0\iint\limits_{\R^{2n}}K(x,y,t)(\zeta(x,t)-\zeta(y,t))^2v^+(y,t)v^+(x,t))dydxdt\ \ \ \ \ \ \ \ \ \ \ \ \ \ \ \ \ \ \ \ \ \ \ \ \ \ \ \ \ \ \ \ \ \ \ \ \ \ \ \ \ \ \ \ \ \ \ \ \ \ $$
$$\ \ \ \ \ \ \ \ \ \ \ \ \ \ \ \ \ \ \ \ \ \ \ \ \ \ \ \ \ \ \ \ \ \ \ \ \ \ \ \ \ \ \ =2\int\limits_{-1}^0\int\limits_{B_1}\int\limits_{\R^{n}}K(x,y,t)(\zeta(x,t)-\zeta(y,t))^2v^+(x,t)v^+(y,t)dxdydt$$
$$\ \ \ \ \ \ \ \ \ \ \ \ \ \ \ \ \ \ \ \ \ \ \ \ \ \ \ \ \ \ \ \ \ \ \ \ \ \leq 2M\Lambda\int\limits_{-1}^0\int\limits_{B_1}\bigg(\int\limits_{\R^n}\frac{(\zeta(x,t)-\zeta(y,t))^2}{|x-y|^{n+\alpha}}dy\bigg)v^+(x,t)dxdydt   $$
$$\leq C(\Lambda,M,n)\int\limits_{-1}^0\int\limits_{B_1}\bigg(\int\limits_{|y-x|>\frac{1}{2}}\frac{2}{|x-y|^{n+\alpha}}dy+\int\limits_{|x-y|\leq\frac{1}{2}}
\frac{(|\grad\zeta(x+s_0(y-x))||x-y||)^2}{|x-y|^{n+\alpha}}dy\bigg)v^+(x,t)dxdt   $$
\begin{equation}\label{hlineq3}
\ \ \ \ \ \ \ \leq C(\Lambda,n,\alpha)\int\limits_{-1}^0\int\limits_{B_1}v^+(x,t)dxdt.
\end{equation}
\end{proof}

In the proof of the Theorem \ref{hthm} below the involvement of the fundamental solution $G$ of our nonlocal operator is crucial. The existence of 
$G$, as well as the corresponding bounds for our case \eqref{Operator} and \eqref{Kbounds} were obtained by Chen and Kumagai in \cite{ChK} (see also Bass and Levin in \cite{BL}). 
%and for the $t$-dependence were obtained by
%Zhen-Qing Chen and Xicheng Zhang in \cite{CZ}. 

The fundamental solution with pole at $(0,0)$ satisfies the two-sided estimates 
$$C^{-1}\Big(t^{-\frac{n}{\alpha}}\wedge \frac{t}{|x|^{n+\alpha}}\Big)\leq G(x,t)\leq C\Big(t^{-\frac{n}{\alpha}}\wedge \frac{t}{|x|^{n+\alpha}}\Big) $$
and the scaling property
$$G(x,t)=t^{-\frac{n}{\alpha}}G\Big(t^{-\frac{1}{\alpha}}x,1\Big).$$
\begin{defn}\label{density} A free boundary point $(x_0,t_0)$ is of parabolic positive density with respect to the coincidence set if there exist positive constants $c>0$ and $r_0>0$ such that $|Q_r(x_0,t_0)\cap\{u=0\}|\geq c|Q_r(x_0,t_0)|$ for every $r<r_0$.
\end{defn}
\begin{thm}\label{hthm} Let $(x_0,t_0)$ be a free boundary point of positive parabolic density with respect to the coincidence set to problem \eqref{derprob}. Then $u_t$ is H\"{o}lder continuous in a neighborhood of $(x_0,t_0)$.
\end{thm}
\begin{proof} For simplicity, we take $(x_0,t_0)=(0,0)$. We penalize the problem \eqref{statem} and consider the corresponding derived equation \eqref{derprob}.  By scaling, it suffices to carry out the estimates in $Q_1$. For any $(\xi,\tau)\in Q_{\frac{1}{5}}$, we multiply the equation by a suitable test function $\eta$ and integrate by parts over the set $Q_{\frac{3}{5}}(\xi,\tau)\subset Q_1$. The test function $\eta$ is a product of three functions i.e.
$$\eta(x,t):=\Big(\zeta^2 G_{\delta}^{(\xi,\tau)}(v^{\varepsilon})^+\Big)(x,t)$$
where $\zeta(x,t)$ is a smooth function supported in $Q_{\frac{3}{5}}(\xi,\tau)$ such that $\zeta\equiv 1$ for every $(x,t)\in Q_{\frac{2}{5}}(\xi,\tau)$, $|\grad\zeta|\leq c$ with $\supp (\grad\zeta)\subset\big(B_{\frac{3}{5}}(\xi,\tau)\setminus B_\frac{2}{5}(\xi,\tau)\big)\times(\tau-(\frac{3}{5})^\alpha,\tau]$ and $0\leq\zeta_t\leq c$ with $\supp (\zeta_t)\subset B_\frac{3}{5} (\xi,\tau)\times (\tau-(\frac{3}{5})^\alpha,\tau-(\frac{2}{5})^\alpha)$, the function $G_\delta^{(\xi,\tau)}(x,t)$ is a $\delta$ smoothing of the fundamental solution of our backward nonlocal equation with pole at $(\xi,\tau)$, i.e. a smooth function such that $G_\delta^{(\xi,\tau)}(x,t)=G(x-\xi.\tau-t)\mathbf{1}_{Q_\delta^c (\xi,\tau)} $ where ${Q_\delta^c (\xi,\tau)}:=\big(\R^n\times(-\infty,\tau]\big)\setminus Q_\delta(\xi,\tau)$. 

Therefore in the weak formulation of \eqref{derprob} we have
\begin{equation}\label{hmaineq}
\int\limits_{-1}^0\int\limits_{\R^n}(\mathcal{L}v^\varepsilon-\partial_t v^\varepsilon)\zeta^2G^{(\xi,\tau)}_\delta(v^\varepsilon)^+ dxdt=\int\limits_{-1}^0 \int\limits_{\R^n}(\beta'_\varepsilon(u^\varepsilon
-\psi^\varepsilon)v^\varepsilon+f_t)\zeta^2G^{(\xi,\tau)}_\delta(v^\varepsilon)^+dxdt.
\end{equation}
In the following calculations, we omit the dependence on 
$\varepsilon$, $\delta$, and $(\xi,\tau)$ for simplicity of notation, and we will restore it when necessary. Hence we obtain
\begin{equation}\label{hineq}
\int\limits_{-1}^0\Bigg[\frac{1}{2}D(v,\zeta^2Gv^+)+\int\limits_{\R^n}((\partial_tv)(\zeta^2Gv^+))(x,t)dx\Bigg]dt\leq -\int\limits_{-1}^0\int\limits_{\R^n}f_t\zeta^2Gv^+dxdt  
\end{equation}
where
$$D(u,v):=\iint\limits_{\R^{2n}}(u(x,t)-u(y,t))K(x,y,t)((v(x,t)-v(y,t))dydx.$$
Observe that
\begin{equation}\label{heq1}D(v,\zeta^2Gv^+)=D(v^+,\zeta^2Gv^+)+D(-v^-,\zeta^2Gv^+).
\end{equation} 
Now
$$D(v^+,\zeta^2Gv^+)=\iint\limits_{\R^{2n}}(v^+(x,t)-v^+(y,t))K(x,y,t)((\zeta^2Gv^+)(x,t)-(\zeta^2Gv^+)(y,t))dydx $$
$$=\iint\limits_{\R^{2n}}(v^+(x,t)-v^+(y,t))K(x,y,t)((\zeta^2Gv^+)(x,t)-G(x,t)(\zeta^2v^+)(y,t))dydx\ \ \ \ \ \ \ \ \   $$
$$\ \ \ \ \ \ \ \ \  +\iint\limits_{\R^{2n}}(v^+(x,t)-v^+(y,t))K(x,y,t)(G(x,t)(\zeta^2v^+)(y,t)-(\zeta^2Gv^+)(y,t))dydx    $$
$$=\iint\limits_{\R^{2n}}K(x,y,t)G(x,t)\Big[(\zeta v^+)^2(x,t)+(\zeta v^+)^2(y,t)-(\zeta^2(x,t)+\zeta^2(y,t))v^+(y,t)v^+(x,t)\Big])dydx\ \ \ \ \ \ \ \ \   $$
$$\ \ \ \ \ \ \ \ \ \ \ \ \ \ \ \ \ \ \ \ \ \  +\iint\limits_{\R^{2n}}(v^+(x,t)-v^+(y,t))K(x,y,t)(\zeta^2v^+)(y,t)(G(x,t)-G(y,t))dydx    $$
$$=\iint\limits_{\R^{2n}}K(x,y,t)G(x,t)\Big[\big((\zeta v^+)(x,t)-(\zeta v^+)(y,t)\big)^2-(\zeta(x,t)-\zeta(y,t))^2v^+(y,t)v^+(x,t)\Big])dydx\ \ \ \ \ \ \ \ \   $$
\begin{equation}\label{heq2}
\ \ \ \ \ \ \ \ \ \ \ \ \ \ \ \ \ \ \  +\iint\limits_{\R^{2n}}(v^+(x,t)(\zeta^2v^+)(y,t)-(\zeta v^+)^2(y,t))K(x,y,t)(G(x,t)-G(y,t))dydx.    
\end{equation} 
Also,
$$D(v^+,\zeta^2Gv^+)=\iint\limits_{\R^{2n}}(v^+(x,t)-v^+(y,t))K(x,y,t)((\zeta^2Gv^+)(x,t)-(\zeta^2Gv^+)(y,t))dydx $$

$$=\iint\limits_{\R^{2n}}(v^+(x,t)-v^+(y,t))K(x,y,t)((\zeta^2Gv^+)(x,t)-G(y,t)(\zeta^2v^+)(x,t))dydx\ \ \ \ \ \ \ \ \   $$
$$\ \ \ \ \ \ \ \ \  +\iint\limits_{\R^{2n}}(v^+(x,t)-v^+(y,t))K(x,y,t)(G(y,t)(\zeta^2v^+)(x,t)-(\zeta^2Gv^+)(y,t))dydx    $$
$$=\iint\limits_{\R^{2n}}(v^+(x,t)-v^+(y,t))K(x,y,t)(\zeta^2 v^+)(x,t)(G(x,t)-G(y,t))dydx\ \ \ \ \ \ \ \ \ \ \ \ \ \ \ \ \ \ \ \ \ \ \ \ \   $$
$$\ \ \ \ \ \ \ \ \  +\iint\limits_{\R^{2n}}K(x,y,t)G(y,t)\Big[(\zeta v^+)^2(x,t)+(\zeta v^+)^2(y,t)-(\zeta^2(x,t)+\zeta^2(y,t))v^+(y,t)v^+(x,t)\Big]dydx    $$
$$=\iint\limits_{\R^{2n}}((\zeta v^+)^2(x,t)-(\zeta^2 v^+)(x,t)v^+(y,t))K(x,y,t)(G(x,t)-G(y,t))dydx\ \ \ \ \ \ \ \ \ \ \ \ \ \ \ \ \ \ \ \ \ \ \ \ \   $$
\begin{equation}\label{heq3}
\ \ \ \ \ \ \ \ \  +\iint\limits_{\R^{2n}}K(x,y,t)G(y,t)\Big[((\zeta v^+)(x,t)-(\zeta v^+)(y,t))^2-(\zeta(x,t)-\zeta(y,t))^2v^+(y,t)v^+(x,t)\Big]dydx.
\end{equation}
Adding (\ref{heq2}) and (\ref{heq3}) and using the symmetry of $K$ we have
$$D(v^+,\zeta^2Gv^+)=\iint\limits_{\R^{2n}}K(x,y,t)G(x,t)((\zeta v^+)(x,t)-(\zeta v^+)(y,t))^2dydx\ \ \ \ \ \ \ \ \ \ \ \ \ \ \ \ \ \ \ \ \ \ \ \ \ \ \     $$
$$\ \ \ \ \ \ \ \ \ \ \ \ \ \ \ \ \ \ \ \ \ \ \ \ -\iint\limits_{\R^{2n}}K(x,y,t)G(x,t)(\zeta(x,t)-\zeta(y,t))^2v^+(y,t)v^+(x,t)dydx   $$
$$+ \frac{1}{2}\iint\limits_{\R^{2n}}((\zeta v^+)^2(x,t)-(\zeta v^+)^2(y,t))K(x,y,t)(G(x,t)-G(y,t))dydx\ \ \ \ \ \ \ \ \ \ \ \ \ \ \ \ \ \ \ \ \ \ \ \ \ \    $$
$$\ \ \ \ \ \ \ \ \ \ \ \ \ \ \ \ \ \ \ \ \ \ \ \ \ \ \ $$
\begin{equation}\label{heq4}
\ \ \ \ \ \ \ \ \ \ \ \ \ \ \ \ \ \ \ \ \ \ \ \ \ \ -\frac{1}{2}\iint\limits_{\R^{2n}} (\zeta^2 (x,t)-\zeta^2(y,t)) v^+(x,t) v^+(y,t))K(x,y,t)(G(x,t)-G(y,t))dydx. 
\end{equation}
The second term in (\ref{heq1})  is
$$D(-v^-,\zeta^2Gv^+)=-\iint\limits_{\R^{2n}}(v^-(x,t)-v^-(y,t))K(x,y,t)((\zeta^2Gv^+)(x,t)-(\zeta^2Gv^+)(y,t))dydx$$
$$\ \ \ \ \ \ \ \ \ \ \ \ \ \ \ \ =-\iint\limits_{\R^{2n}}\Big[v^-(x,t)(\zeta^2Gv^+)(x,t)+v^-(y,t)(\zeta^2Gv^+)(y,t)\Big]K(x,y,t)dydx$$
$$\ \ \ \ \ \ \ \ \ \ \ \ \ \ \ \ \ \ \ \ \ \ \ \ \ \ \ \ \ \ \ \ \ \ \ \ \ \ \ \ \ \ +\iint\limits_{\R^{2n}}\Big[v^-(y,t)(\zeta^2Gv^+)(x,t)+v^-(x,t)(\zeta^2Gv^+)(y,t)\Big]K(x,y,t)dydx.   $$
Since the first integral above is zero and the second one is nonnegative, the term $D(-v^-,\zeta^2Gv^+)\geq0$ and will be ignored in the inequality (\ref{hineq}).

The second term of the left hand side of (\ref{hineq})
$$\int\limits_{-1}^0\int\limits_{\R^n}((\partial_tv)(\zeta^2Gv^+))(x,t)dxdt=\int\limits_{-1}^0\int\limits_{\R^n}G\zeta v^+(\partial_t(\zeta v^+)-v^+\partial_t\zeta)dxdt\ \ \ \ \ \ \ \ \ \ \ \ \ \ \ \ \ \ \ \ \ \ \ \ \ \  $$
\begin{equation}\label{heq5}
\ \ \ \ \ \ \ \ \ \ \ \ \ \ \ \ \ \ \ \ \ \ \ \ \ \ \ \ \ \ \ \ \ \ \ \ \ \ \ \ =\int\limits_{-1}^0\int\limits_{\R^n}\frac{1}{2}G\partial_t\Big((\zeta v^+)^2\Big)dxdt-\int\limits_{-1}^0\int\limits_{\R^n}\zeta\zeta_t (v^+)^2Gdxdt.
\end{equation}

Inserting (\ref{heq4}) and (\ref{heq5}) into inequality (\ref{hineq}) and keeping the terms with $+$ sign on the left hand side of the inequality and the ones with $-$ sign on the right hand side we obtain
$$\int\limits_{-1}^0\iint\limits_{\R^{2n}}K(x,y,t)G(x,t)((\zeta v^+)(x,t)-(\zeta v^+)(y,t))^2dydxdt$$ 
$$+\frac{1}{4}\int\limits_{-1}^0\iint\limits_{\R^{2n}}((\zeta v^+)^2(x,t)-(\zeta v^+)^2(y,t))K(x,y,t)(G(x,t)-G(y,t))dydx+\frac{1}{2}\int\limits_{-1}^0\int\limits_{\R^n}G\partial_t\Big((\zeta v^+)^2\Big)dxdt   $$
$$\ \ \ \ \ \ \ \ \ \leq\int\limits_{-1}^0\iint\limits_{\R^{2n}}K(x,y,t)G(x,t)(\zeta(x,t)-\zeta(y,t))^2v^+(y,t)v^+(x,t)dydxdt$$
$$\ \ \ \ \ \ \ \ \ \ \ \ \ \ \ \ \ \ \ \ \ \ \ \ \ \ \ \ \ \ +\frac{1}{4}\int\limits_{-1}^0\iint\limits_{\R^{2n}} (\zeta^2 (x,t)-\zeta^2(y,t)) v^+(x,t) v^+(y,t))K(x,y,t)(G(x,t)-G(y,t))dydxdt$$ 
\begin{equation}\label{hineq1}
+\int\limits_{-1}^0\int\limits_{\R^n}\zeta\zeta_t (v^+)^2Gdxdt -\int\limits_{-1}^0\int\limits_{\R^n}f_t\zeta^2Gv^+dxdt.\ \ \ \ \ \ \ \ \ \ \ \ \ \ \ \ \     \end{equation}
Using the lower bound of the kernel $K$, the properties of the cut-off function $\zeta$ and the fact that $G$ is the smoothing of the fundamental solution of our nonlocal equation the left-hand of \eqref{heq1} reduces to
$$\frac{1}{\Lambda}\int\limits_{\tau-(\frac{1}{5})^\alpha}^\tau\int\limits_{B_{\frac{1}{5}}(\xi)}\int\limits_ {B_{\frac{1}{5}}(\xi)}G(x,t)\frac{((v^+)(x,t)-(v^+)(y,t))^2}{|x-y|^{n+\alpha}}dydxdt+\fint_{Q_\delta (\xi,\tau)}(v^+)^2dxdt\ \ \ \ \ \ \ \ \ \ \ \ \ \ \ \ \ \ \ \ \ \ \ \ \ \ \   $$
$$\ \ \ \ \ \ \ \leq 2\int\limits_{-1}^0\iint\limits_{\R^{2n}}K(x,y,t)G(x,t)(\zeta(x,t)-\zeta(y,t))^2v^+(y,t)v^+(x,t)dydxdt$$
$$\ \ \ \ \ \ \ \ \ \ \ \ \ \ \ \ \ \ \ \ \ \ \ +\frac{1}{2}\int\limits_{-1}^0\iint\limits_{\R^{2n}} (\zeta^2 (x,t)-\zeta^2(y,t)) v^+(x,t) v^+(y,t))K(x,y,t)(G(x,t)-G(y,t))dydxdt$$ 
\begin{equation}\label{hineq2}
+2\int\limits_{-1}^0\int\limits_{\R^n}\zeta\zeta_t (v^+)^2Gdxdt-2\int\limits_{-1}^0\int\limits_{\R^n}f_t\zeta^2Gv^+dxdt.\ \ \ \ \ \ \ \ \ \ \ \ \ \ \ \ \ \ 
\end{equation} 

Now we estimate the four terms on the right-hand side of \eqref{hineq2}.
We will use the bounds of the fundamental solution, namely
$0\leq G_\delta^{(\xi,\tau)}\leq C(n,\alpha)$ in $(B_\frac{4}{5}(\xi)\setminus B_\frac{1}{5})(\xi)\times(\tau-(\frac{2}{5})^\alpha,\tau)$, and $c(n,\alpha)\leq G_\delta^{(\xi,\tau)}\leq C(n,\alpha)$ in $B_\frac{4}{5}(\xi)\times(\tau-(\frac{3}{5})^\alpha,\tau-(\frac{2}{5})^\alpha)$.  The first term
$$2\int\limits_{-1}^0\iint\limits_{\R^{2n}}G(x,t)K(x,y,t)
(\zeta(x,t)-\zeta(y,t))^2v^+(x,t)v^+(y,t)dydxdt\ \ \ \ \ \ \ \ \ \ \ \ \ \ \ \ \ \ \ \ \ \ \ \ \ \ \ \ \ \ \ \ \ \ \ \ \ \ $$
$$\leq\int\limits_{-1}^0\iint\limits_{\R^{2n}}G(x,t)K(x,y,t)
(\zeta(x,t)-\zeta(y,t))^2\Big[(v^+(x,t))^2+(v^+(y,t))^2\Big]dydxdt\ \ \ \ \ \ \ \ $$
$$ \ \ \ \ \ \ \ \ \ \ \ \ \leq 2C(n,\alpha)\int\limits_{\tau-(\frac{3}{5})^\alpha}^\tau\int\limits_{B_\frac{3}{5}(\xi)\setminus B_\frac{1}{5}(\xi)}(v^+(x,t))^2\int\limits_{B_\frac{4}{5}\setminus B_\frac{1}{5}}(\zeta(x,t)-\zeta(y,t))^2K(x,y,t)dydxdt\ \ \ \ \ \ \ \ \ \ \ \ \ \ $$
$$\ \ \ \ \ \ \ \ \ \ \ \ \ \ \ \ \ \ \leq 2C(n,\alpha)\Lambda\int\limits_{\tau-(\frac{3}{5})^\alpha}^\tau\int\limits_{B_\frac{3}{5}(\xi)\setminus B_\frac{1}{5}(\xi)}(v^+(x,t))^2\bigg(\int\limits_{|x-y|\leq\frac{1}{2}}
\frac{|\grad\zeta(x+s_0(y-x)\cdot(y-x)|^2}{|x-y|^{n+\alpha}}dy\bigg)dxdt\ \ \ \  $$ 
$$\ \ \ \ \ \ \ \ \ \ \ \ \ \ \ \ \ \ \ \ \ \ \ \ \ \ \ \ \ \ \ \ \ \ \ \ \ \ \ \ \ \ \ \ \ \ \ \ \ \ \ \ \ \ \ \ \ \ \ \ \ \ \ \ \ \ \ \ \ \ \ \ \ \ \ \ \ \ \ \ \ \ \ \ \  \ \ \ \ \ \ \ \ \ \ \ \ \ \ \ \ \ \ \ \ \ \ \text{for} \ \ s_0\in(0,1)$$
\begin{equation}\label{hineq3}\leq 2C(n,\alpha)\Lambda\int\limits_{\tau-(\frac{3}{5})^\alpha}^\tau\int\limits_{B_\frac{3}{5}(\xi)\setminus B_\frac{1}{5}(\xi)}((v^+)(x,t))^2dx.\ \ \ \ \ \ \ \ \ \ \ \ \ \ \ \ \ \ \ \ \ \ \ \ \ \ \ \ \ \ \ \ \ \ \ \ \ \ \  
\end{equation}
For the second term on the right hand side of (\ref{hineq2}) we have 
$$\frac{1}{2}\int\limits_{-1}^0\iint\limits_{\R^{2n}} (\zeta^2 (x,t)-\zeta^2(y,t)) v^+(x,t) v^+(y,t))K(x,y,t)(G(x,t)-G(y,t))dydxdt\ \ \ \ \ \ \ \ \ \ \ \ \ \ \ \ \ \ \ \ \ \ \ \ \ \ \ $$
$$\leq\frac{1}{4}\int\limits_{-1}^0\iint_{\R^{2n}}(\zeta^2(x,t)-\zeta^2(y,t))((v^+(x,t))^2+(v^+(y,t))^2)K(x,y,t)(G(x,t)-G(y,t))dydxdt\ \ $$
$$\ \ \ \ \ =\frac{1}{2}\int\limits_{\tau-(\frac{3}{5})^\alpha}^\tau\int\limits_{B_{\frac{3}{5}}(\xi)\setminus B_{\frac{1}{5}}(\xi)}(v^+(x,t))^2\int\limits_{B_{\frac{4}{5}}(\xi)\setminus B_{\frac{1}{5}}(\xi)}(\zeta^2(x,t)-\zeta^2(y,t))K(x,y,t)(G(x,t)-G(y,t))dydxdt   $$
$$\leq\int\limits_{\tau-(\frac{3}{5})^\alpha}^\tau\int\limits_{B_{\frac{3}{5}}(\xi)\setminus B_{\frac{1}{5}}(\xi)}(v^+(x,t))^2\int\limits_{B_{\frac{4}{5}}(\xi)\setminus B_{\frac{1}{5}}(\xi)}(\zeta(x,t)-\zeta(y,t))K(x,y,t)(G(x,t)-G(y,t))dydxdt$$
$$\leq C(n,\alpha)\Lambda\int\limits_{\tau-(\frac{3}{5})^\alpha}^\tau\int\limits_{B_{\frac{3}{5}}(\xi)\setminus B_{\frac{1}{5}}(\xi)}(v^+(x,t))^2\ \ \ \ \ \ \ \ \ \ \ \ \ \ \ \ \ \ \ \ \ \ \ \ \ \ \ \ \ \ \ \ \ \ \ \ \ \ \ \ \ \ \ \ \ \ \ \ \ \ \ \ \ \ \ \ \ \ \ \ \ \ \ \ \ \ \ $$
$$\ \ \ \ \ \ \ \ \ \ \ \ \ \ \ \ \ \ \ \ \ \ \ \ \ \ \ \ \ \ \ \ \ \ \ \ \ \ \ \ \ \ \ \ \ \ \ \ \ \ \ \ \ \ \times\int\limits_{B_{\frac{4}{5}}(\xi)\setminus B_{\frac{1}{5}}(\xi)}\frac{(\grad\zeta(x+s_0(y-x))\cdot(y-x)(\grad G(x+s_1(y-x))\cdot(y-x)}{|x-y|^{n+\alpha}}dydxdt$$
$$\leq C(n,\alpha)\Lambda\int\limits_{\tau-(\frac{3}{5})^\alpha}^\tau\int\limits_{B_{\frac{3}{5}}(\xi)\setminus B_{\frac{1}{5}}(\xi)}(v^+(x,t))^2\int\limits_{B_{\frac{4}{5}}(\xi)\setminus B_{\frac{1}{5}}(\xi)}\frac{|y-x|^2}{|x-y|^{n+\alpha}}dydxdt\ \ \ \ \ \ \ \ \ \ \ \ \ \ \ \ \ \ \ \ \ \ \ \ \ \ \ \ $$
\begin{equation}\label{hineq4}
\leq C(n,\alpha)\Lambda\int\limits_{\tau-(\frac{3}{5})^\alpha}^\tau\int\limits_{B_\frac{3}{5}(\xi)\setminus B_{\frac{1}{5}}(\xi)}(v^+(x,t))^2dxdt.\ \ \ \ \ \ \ \ \ \ \ \ \ \ \ \ \ \ \ \ \ \ \ \ \ \ \ \ \ \ \ \ \ \ \ \ \ \ \ \ \ \ \ \ \ \ \ \ \ \ \ \  
\end{equation}
The third term on the right hand of (\ref{hineq2}) becomes
$$2\int\limits_{-1}^0\int\limits_{\R^n}\zeta\zeta_t (v^+)^2Gdxdt\leq2C(n,\alpha)\int\limits_{-\frac{2}{5}}^{-\frac{4}{25}}\int\limits_{B_\frac{4}{5}}(v^+(x,t))^2dxdt$$
and, finally, the fourth term
$$-2\int\limits_{-1}^0\int\limits_{\R^n}f_t\zeta^2Gv^+dxdt\leq C(n,\alpha)M.$$

Substituting the estimates just obtained into (\ref{hineq2}) we have
$$\int\limits_{\tau-(\frac{1}{5})^\alpha}^\tau\int\limits_{B_{\frac{1}{5}}(\xi)}\int\limits_ {B_{\frac{1}{5}}(\xi)}G^{(\xi,\tau)}_\delta(x,t)\frac{((v^\varepsilon)^+(x,t)-(v^\varepsilon)^+(y,t))^2}{|x-y|^{n+\alpha}}dydxdt+\fint_{Q_{\delta}(\xi,\tau)}((v^\varepsilon)^+)^2dxdt\ \ \ \ \ \ \ \ \ \ \ \ \ \ \ \ \ \ \ \ \ \ \ \ \ \ \ \ \ \ \ \ \ \ \ \ \ \ \ \ \ \ \ \ \ $$ 
\begin{equation}\label{hineq5}\ \ \ \ \ \ \ \ \ \ \ \ \ \ \ \ \ \ \ \ \ \ \ \ \ \ \ \ 
\leq C(n,\alpha,\Lambda)(\int\limits_{-(\frac{3}{5})^\alpha}^0\int\limits_{B_\frac{3}{5}\setminus B_{\frac{1}{5}}}((v^\varepsilon)^+(x,t))^2dxdt+\int\limits_{-(\frac{3}{5})^\alpha}^{-(\frac{2}{5})^\alpha}\int\limits_{B_\frac{4}{5}}((v^\varepsilon)^+(x,t))^2dxdt+M).
\end{equation}
Now, we first let $\varepsilon$ go to $0$ to obtain \eqref{hineq5} for $v^+$, then we $\delta$ to go to $0$, and finally we take the supremum over $(\xi,\tau)\in Q^+_\frac{1}{5}$ to obtain
$$\sup_{Q_\frac{1}{5}}\  (v^+)^2+\int\limits_{(\frac{1}{5})^\alpha}^0\int\limits_{B_\frac{1}{5}}\int\limits_{B_\frac{1}{5}}G(x,-t)\frac{((v^+)(x,t)-(v^+)(y,t))^2}{|x-y|^{n+\alpha}}dydxdt \ \ \ \ \ \ \ \ \ \ \ \ \ \ \ \ \ \ \ \ \ \ \ \ \ \ \ \ \ \ \ \ \ \ \ \ \ \ \ \ \ \ \ \ \ \ \ \ \ \ \ \ \ \ \ \ \ \ \ \ \ \  $$
$$\ \ \ \ \ \ \ \ \ \ \ \ \ \ \ \ \ \ \ \ \ \ \ \ \ \ \ \ \ \ \ \leq C(n,\alpha,\Lambda)(\int\limits_{-(\frac{3}{5})^\alpha}^0\int\limits_{B_\frac{3}{5}\setminus B_{\frac{1}{5}}}(v^+(x,t))^2dxdt+\int\limits_{-(\frac{3}{5})^\alpha}^{-(\frac{2}{5})^\alpha}\int\limits_{B_\frac{4}{5}}(v^+(x,t))^2dxdt+M)$$
\begin{equation}\label{hineq6}
 \ \ \ \ \ \ \ \ \ \ \ \ \ \ \leq C(n,\alpha,\Lambda)\bigg(\sup_{Q_1\setminus Q_{\frac{1}{5}}}(v^+)^2+\int\limits_{-(\frac{3}{5})^\alpha}^{-(\frac{2}{5})^\alpha}\int\limits_{B_\frac{4}{5}}(v^+(x,t))^2dxdt+M\bigg).  
\end{equation}
We next replace the integrand of the second term on the right-hand side of \eqref{hineq6} by an integrand of the first term on the left-hand side
of \eqref{hineq6}. By Lemma \ref{hlem1} and the density assumption we can apply Poincar\'{e} inequality (see \cite{FK} and \cite{FKV}) to have
\begin{equation}\label{hineq7}
\int\limits_{-(\frac{3}{5})^\alpha}^{-(\frac{2}{5})^\alpha}\int\limits_{B_\frac{4}{5}}(v^+(x,t))^2dxdt\leq C(n,\alpha)\int\limits_{-(\frac{3}{5})^\alpha}^{-(\frac{2}{5})^\alpha}\int\limits_{B_\frac{4}{5}}\int\limits_{B_\frac{4}{5}}\frac{((v^+)(x,t)-(v^+)(y,t))^2}{|x-y|^{n+\alpha}}dydxdt
\end{equation}
and since $G(x,-t)\geq c(n,\alpha)$ in the domain of integration we obtain
$$\ \ \ \ \ \ \ \ \ \ \ \ \ \ \ \ \ \ \ \ \ \ \ \ \ \ \ \ \ \ \ \ \ \ \ \  \leq C(n,\alpha)\int\limits_{-(\frac{3}{5})^\alpha}^{-(\frac{2}{5})^\alpha}\int\limits_{B_\frac{4}{5}}\int\limits_{B_\frac{4}{5}}G(x,-t)\frac{((v^+)(x,t)-(v^+)(y,t))^2}{|x-y|^{n+\alpha}}dydxdt.$$
Inserting the above in \eqref{hineq6} we finally have
$$\sup_{Q_\frac{1}{5}}\ (v^+)^2 +\int\limits_{(\frac{1}{5})^\alpha}^0\int\limits_{B_\frac{1}{5}}\int\limits_{B_\frac{1}{5}}G(x,-t)\frac{((v^+)(x,t)-(v^+)(y,t))^2}{|x-y|^{n+\alpha}}dydxdt\ \ \ \ \ \ \ \ \ \ \ \ \ \ \ \ \ \ \ \ \ \ \ \ \ \ \ \ \ $$ 
\begin{equation}\label{hineq8}\ \ \ \ \ \ \ \ \ \ \ \ \ \ \ \ \ \ \ \ \ \ \ \ \ \ \leq C(n,\alpha,\Lambda)\bigg(\sup_{Q_1\setminus Q_{\frac{1}{5}}}(v^+)^2+\int\limits_{-(\frac{3}{5})^\alpha}^{-(\frac{2}{5})^\alpha}\int\limits_{B_\frac{4}{5}}\int\limits_{B_\frac{4}{5}}G(x,-t)\frac{((v^+)(x,t)-(v^+)(y,t))^2}{|x-y|^{n+\alpha}}dydxdt\bigg). 
\end{equation}
Now, we multiply the left-hand side of \eqref{hineq8} by the constant $C(n.\alpha,\Lambda)$ of the right-hand side and we add this term to \eqref{hineq8} to obtain 
$$\sup_{Q_\frac{1}{5}}\ (v^+)^2 +\int\limits_{-(\frac{1}{5})^\alpha}^0\int\limits_{B_\frac{1}{5}}\int\limits_{B_\frac{1}{5}}G(x,-t)\frac{((v^+)(x,t)-(v^+)(y,t))^2}{|x-y|^{n+\alpha}}dydxdt\ \ \ \ \ \ \ \ \ \ \ \ \ \ \ \ \ \ \ \ \ \ \ \ \ \ \ \ \ $$ 
\begin{equation}\label{hineq9}\ \ \ \ \ \ \ \ \ \ \ \ \ \ \ \ \ \ \ \ \ \ \ \ \ \ \leq \frac{C(n,\alpha,\Lambda)}{1+C(n,\alpha,\Lambda)}\bigg(\sup_{Q_1\setminus Q_{\frac{1}{5}}}(v^+)^2+\int\limits_{-1}^0\int\limits_{B_1}\int\limits_{B_1}G(x,-t)\frac{((v^+)(x,t)-(v^+)(y,t))^2}{|x-y|^{n+\alpha}}dydxdt\bigg) 
\end{equation}
where in the integral of the right-hand side we have replaced the domain of integration by a larger one.

By setting $$\omega(\rho):=\sup_{Q_\rho} (v^+)^2 +\int\limits_{-(\rho)^\alpha}^0\int\limits_{B_\rho}\int\limits_{B_\rho}G(x,-t)\frac{((v^+)(x,t)-(v^+)(y,t))^2}{|x-y|^{n+\alpha}}dydxdt$$ we write \eqref{hineq9} as
\begin{equation}\label{hineq10}
\omega (\frac{1}{5})\leq\lambda\omega(1)+M
\end{equation}
where $\lambda:=\frac{C(n,\alpha,\Lambda)}{1+C(n,\alpha,\Lambda)}$. Iteration of \eqref{hineq10} implies that there exists a $\gamma=\gamma(\lambda)\in(0,1)$ and a constant $C=(n,\alpha,\Lambda,\|{u_t}\|_\infty,\|f_t\|)$ such that 
$$\omega(\rho)\leq C\rho^\gamma$$ for every $0\leq\rho\leq\frac{r_0}{5}$ where  $r_0$ is as in Definition~\ref{density}. Replacing $v^+$ by $v^-$ in the above calculations yields the same estimate, and this implies the H\"{o}lder continuity of $u_t$.
\end{proof}

\bibliographystyle{plain}   % Here the bibliography
\bibliography{biblio}             % is inserted.
\index{Bibliography@\emph{Bibliography}}%

\vspace{2em}

\begin{tabular}{l}
Ioannis Athanasopoulos\\ University of Crete \\ Department of Mathematics  \\ 71409 \\
Heraklion, Crete GREECE
\\ {\small \tt athan@uoc.gr}
\end{tabular}
\begin{tabular}{l}
Luis Caffarelli\\ University of Texas \\ Department of Mathematics  \\ TX 78712\\
Austin, USA
\\ {\small \tt caffarel@math.utexas.edu}
\end{tabular}
\begin{tabular}{lr}
Emmanouil Milakis\\ University of Cyprus \\ Department of Mathematics \& Statistics \\ P.O. Box 20537\\
Nicosia, CY- 1678 CYPRUS
\\ {\small \tt emilakis@ucy.ac.cy}
\end{tabular}

\end{document}